\documentclass{amsproc}
\usepackage{amsfonts}

\theoremstyle{plain}

\newtheorem{definition}{Definition}

\newtheorem{remark}{Remark}

\newtheorem{theorem}{Theorem}
\numberwithin{equation}{section}
\input{tcilatex}

\begin{document}
\title[On $MT-$Convexity]{On $MT-$Convexity}
\author{Mevl\"{u}t TUN\c{C}}
\address{Department of Mathematics, Faculty of Art and Sciences, Kilis 7
Aralik University, Kilis, 79000, Turkey}
\email{mevluttunc@kilis.edu.tr}
\author{H\"{u}seyin YILDIRIM}
\address{Department of Mathematics, Faculty of Art and Sciences, Kahramanmara%
\c{s} S\"{u}t\c{c}\"{u} \.{I}mam University, Kahramanmara\c{s}, 46000, Turkey%
}
\email{hyildir@ksu.edu.tr}
\subjclass[2000]{26D15}
\keywords{convexity, AM-GM inequality.}

\begin{abstract}
In this paper, one new classes of convex functions which is called $MT-$%
convex functions are given. We also establish some Hadamard-type
inequalities.
\end{abstract}

\maketitle

\section{Introduction}

The following definition is well known in the literature: A function $%
f:I\rightarrow 
\mathbb{R}
,$ $\emptyset \neq I\subseteq 
\mathbb{R}
,$ is said to be convex on $I$ if inequality

\begin{equation}
f\left( tx+\left( 1-t\right) y\right) \leq tf\left( x\right) +\left(
1-t\right) f\left( y\right)  \label{co}
\end{equation}%
holds for all $x,y\in I$ and $t\in \left[ 0,1\right] $. Geometrically, this
means that if $P,Q$ and $R$ are three distinct points on the graph of $f$
with $Q$ between $P$ and $R$, then $Q$ is on or below chord $PR$.

Let $f:I\subseteq 
\mathbb{R}
\rightarrow 
\mathbb{R}
$ be a convex function and $a,b\in I$ with $a<b$. The following double
inequality:$\ $

\begin{equation}
f\left( \frac{a+b}{2}\right) \leq \frac{1}{b-a}\int_{a}^{b}f\left( x\right)
dx\leq \frac{f\left( a\right) +f\left( b\right) }{2}  \label{h-h}
\end{equation}%
is known in the literature as Hadamard's inequality (or H-H inequality) for
convex function. Keep in mind that some of the classical inequalities for
means can come from (\ref{h-h}) for convenient particular selections of the
function $f$. If $f$ is concave, this double inequality hold in the inversed
way.

In \cite{p}, Pachpatte established two Hadamard-type inequalities for
product of convex functions.

\begin{theorem}
Let $f,g:[a,b]\subseteq \mathbb{R}\rightarrow \lbrack 0,\infty )$ be convex
functions on $[a,b]$, $a<b$. Then%
\begin{equation}
\frac{1}{b-a}\int_{a}^{b}f(x)g(x)dx\leq \frac{1}{3}M(a,b)+\frac{1}{6}N(a,b)
\label{PA}
\end{equation}%
and%
\begin{equation*}
2f\left( \frac{a+b}{2}\right) g\left( \frac{a+b}{2}\right) \leq \frac{1}{b-a}%
\int_{a}^{b}f(x)g(x)dx+\frac{1}{6}M(a,b)+\frac{1}{3}N(a,b)
\end{equation*}%
where $M(a,b)=f(a)g(a)+f(b)g(b)$ and $N(a,b)=f(a)g(b)+f(b)g(a)$.
\end{theorem}

\bigskip We recall the well-known AM-GM inequality for n positive real
numbers which can be stated as follows.

\bigskip If $p_{1},...,p_{n}$ are positive numbers which sum to 1 and\ $f$\
is a real continuous function that is concave up, then%
\begin{equation*}
\sum_{i=1}^{n}p_{i}f\left( x_{i}\right) \geq f\left(
\sum_{i=1}^{n}p_{i}x_{i}\right)
\end{equation*}%
A special case is 
\begin{equation*}
\sqrt[n]{x_{1}x_{2}...x_{n}}\leq \frac{x_{1}+x_{2}+...+x_{n}}{n}
\end{equation*}%
with equality iff $x_{1}=x_{2}=...=x_{n}.$

\bigskip We recall the well-known AM-GM inequality for two positive real
numbers which can be stated as follows.

If $x,y\in 
\mathbb{R}
^{+},$ then%
\begin{equation}
\sqrt{xy}\leq \frac{x+y}{2}.  \label{ag}
\end{equation}%
with equality if and only if $x=y.$

This inequality has many simple proofs. For example, a proof based on the
concavity of the logarithmic function is presented in various sources, and
the original reference is Jensen's paper \cite{j}. A proof based on
induction, given by Cauchy in 1821, is presented in many sources, as for
example in \cite{b}, pp. 1--2.

In the following section our main results are given. We establish new a
class of convex functions and then we obtain new Hadamard type inequalities
for the new class of convex function.

\begin{definition}
\lbrack See \cite{A}] Two functions $f:X\rightarrow 
\mathbb{R}
$ and $g:X\rightarrow 
\mathbb{R}
$ are said to be similarly ordered, shortly $f$ s.o. $g$, if%
\begin{equation*}
(f(x)-f(y))(g(x)-g(y))\geq 0
\end{equation*}%
for every $x,y\in X.$
\end{definition}

\section{$MT-$Convexity and Related Results}

\begin{remark}
\bigskip If we take $x=t$ and $y=1-t$ in (\ref{ag}), we have%
\begin{equation}
1\leq \frac{1}{2\sqrt{t\left( 1-t\right) }}  \label{ga}
\end{equation}%
for all $t\in \left( 0,1\right) .$
\end{remark}

\begin{definition}
A function $f:I\subseteq 
\mathbb{R}
\rightarrow 
\mathbb{R}
$ is said to belong to the class of $MT\left( I\right) $ if it is
nonnegative and for all $x,y\in I$ and $t\in \left( 0,1\right) $ satisfies
the inequality;%
\begin{equation}
f\left( tx+\left( 1-t\right) y\right) \leq \frac{\sqrt{t}}{2\sqrt{1-t}}%
f\left( x\right) +\frac{\sqrt{1-t}}{2\sqrt{t}}f\left( y\right) .  \label{1}
\end{equation}
\end{definition}

\begin{remark}
In (\ref{1}), if we take $t=1/2$, inequality (\ref{1}) reduce to Jensen
convex.
\end{remark}

\begin{theorem}
Let $f\in MT\left( I\right) $ , $a,b\in I$ with $a<b$ and $f\in L_{1}\left[
a,b\right] .$Then%
\begin{equation}
f\left( \frac{a+b}{2}\right) \leq \frac{1}{b-a}\int_{a}^{b}f\left( x\right)
dx  \label{2}
\end{equation}%
and%
\begin{equation}
\frac{2}{b-a}\int_{a}^{b}\tau \left( x\right) f\left( x\right) dx\leq \frac{%
f\left( a\right) +f\left( b\right) }{2}  \label{3}
\end{equation}%
where $\tau \left( x\right) =\frac{\sqrt{\left( b-x\right) \left( x-a\right) 
}}{b-a},$ $x\in \left[ a,b\right] $
\end{theorem}

\begin{proof}
Since $f\in MT\left( I\right) $, we have, for all, $x,y\in I$ (with $t=\frac{%
1}{2}$ in (\ref{1})) that%
\begin{equation*}
f\left( \frac{x+y}{2}\right) \leq \frac{f\left( x\right) +f\left( y\right) }{%
2}
\end{equation*}%
i.e. with $x=ta+\left( 1-t\right) b,$ $y=\left( 1-t\right) a+tb,$%
\begin{equation*}
f\left( \frac{a+b}{2}\right) \leq \frac{1}{2}\left( f\left( ta+\left(
1-t\right) b\right) +f\left( \left( 1-t\right) a+tb\right) \right) .
\end{equation*}%
By integrating, we get%
\begin{equation}
f\left( \frac{a+b}{2}\right) \leq \frac{1}{2}\left[ \int_{0}^{1}f\left(
ta+\left( 1-t\right) b\right) dt+\int_{0}^{1}f\left( \left( 1-t\right)
a+tb\right) dt\right] ,  \label{4}
\end{equation}%
Since%
\begin{equation*}
\int_{0}^{1}f\left( ta+\left( 1-t\right) b\right) dt=\int_{0}^{1}f\left(
\left( 1-t\right) a+tb\right) dt=\frac{1}{b-a}\int_{a}^{b}f\left( u\right)
du,
\end{equation*}%
we get the inequality (\ref{2}) from (\ref{4}).

For the proof of (\ref{3}), we first note that if $f\in MT\left( I\right) $,
then for all $a,b\in I$ and $t\in \left[ 0,1\right] $, it yields%
\begin{equation*}
2\sqrt{t\left( 1-t\right) }f\left( ta+\left( 1-t\right) b\right) \leq
tf\left( a\right) +\left( 1-t\right) f\left( b\right)
\end{equation*}%
and%
\begin{equation*}
2\sqrt{t\left( 1-t\right) }f\left( \left( 1-t\right) a+tb\right) \leq \left(
1-t\right) f\left( a\right) +tf\left( b\right) .
\end{equation*}%
By adding these inequalities and integrating on $t$ over $\left[ 0,1\right] $%
, we obtain%
\begin{equation}
\int_{0}^{1}\sqrt{t\left( 1-t\right) }\left[ f\left( ta+\left( 1-t\right)
b\right) +f\left( \left( 1-t\right) a+tb\right) \right] dt\leq \frac{f\left(
a\right) +f\left( b\right) }{2}  \label{5}
\end{equation}%
Therefore,%
\begin{eqnarray}
\int_{0}^{1}\sqrt{t\left( 1-t\right) }f\left( ta+\left( 1-t\right) b\right)
dt &=&\int_{0}^{1}\sqrt{t\left( 1-t\right) }f\left( \left( 1-t\right)
a+tb\right) dt  \label{6} \\
&=&\frac{1}{b-a}\int_{0}^{1}\frac{\sqrt{\left( b-x\right) \left( x-a\right) }%
}{\left( b-a\right) ^{2}}f\left( x\right) dx.  \notag
\end{eqnarray}%
We get (\ref{3}) by combining (\ref{5}) with (\ref{6}) and the proof is
completed.

The constant $1$ in (\ref{2}) is the best possible because this inequality
obviously reduces to an equality for the function $f\left( x\right) =1$ for
all $a\leq x\leq b.$ Additionally, this function is to in the class $%
MT\left( I\right) $, because%
\begin{eqnarray*}
\frac{\sqrt{t}}{2\sqrt{\left( 1-t\right) }}f\left( x\right) +\frac{\sqrt{%
\left( 1-t\right) }}{2\sqrt{t}}f\left( y\right) &\geq &\frac{\sqrt{t}}{2%
\sqrt{\left( 1-t\right) }}+\frac{\sqrt{\left( 1-t\right) }}{2\sqrt{t}}%
=g\left( t\right) \\
&\geq &%
\begin{array}{cc}
\min g\left( t\right) =g\left( \frac{1}{2}\right) & 0<t<1%
\end{array}
\\
&=&1\geq f\left( tx+\left( 1-t\right) y\right)
\end{eqnarray*}%
for all $x,y\in \left[ a,b\right] $ and $t\in \left[ 0,1\right] .$Thus, the
proof is completed.
\end{proof}

\begin{remark}
In (\ref{3}), if we take $x=\frac{a+b}{2}$, inequality (\ref{3}) reduce to
Jensen's inequality.
\end{remark}

\begin{theorem}
Let $f\in MT\left( I\right) $ , $a,b\in I$ with $a<b$ and $f\in L_{1}\left[
a,b\right] .$Then%
\begin{equation*}
\frac{\pi }{2}f\left( \frac{a+b}{2}\right) \leq f\left( a\right) +f\left(
b\right) .
\end{equation*}
\end{theorem}

\begin{proof}
Since $f\in MT\left( I\right) $, we have%
\begin{eqnarray*}
f\left( \frac{a+b}{2}\right) &\leq &f\left( \frac{ta+\left( 1-t\right) b}{2}+%
\frac{\left( 1-t\right) a+tb}{2}\right) \\
&\leq &\frac{1}{2}\left( f\left( ta+\left( 1-t\right) b\right) +f\left(
\left( 1-t\right) a+tb\right) \right) \\
&\leq &\frac{1}{2}\left( \frac{\sqrt{t}}{2\sqrt{\left( 1-t\right) }}+\frac{%
\sqrt{\left( 1-t\right) }}{2\sqrt{t}}\right) \left( f\left( a\right)
+f\left( b\right) \right)
\end{eqnarray*}%
Moreover%
\begin{equation*}
4\sqrt{t\left( 1-t\right) }f\left( \frac{a+b}{2}\right) \leq f\left(
a\right) +f\left( b\right)
\end{equation*}%
By integrating, we get%
\begin{eqnarray*}
4f\left( \frac{a+b}{2}\right) \int_{0}^{1}\sqrt{t\left( 1-t\right) }dt &\leq
&f\left( a\right) +f\left( b\right) \\
\frac{\pi }{2}f\left( \frac{a+b}{2}\right) &\leq &f\left( a\right) +f\left(
b\right) .
\end{eqnarray*}%
The proof is completed.
\end{proof}

\bigskip

\begin{theorem}
\bigskip Let $f,g\in MT\left( I\right) $ , $a,b\in I$ with $a<b$ and $fg\in
L_{1}\left[ a,b\right] .$Then we have the inequality%
\begin{eqnarray}
&&\frac{1}{b-a}\int_{a}^{b}\mu \left( x\right) f\left( x\right) g\left(
x\right) dx  \label{x} \\
&\leq &\frac{1}{12}\left[ f\left( a\right) g\left( a\right) +f\left(
b\right) g\left( b\right) \right] +\frac{1}{24}\left[ f\left( a\right)
g\left( b\right) +f\left( b\right) g\left( a\right) \right]  \notag
\end{eqnarray}%
where $\mu \left( x\right) =\frac{\left( b-x\right) \left( x-a\right) }{%
\left( b-a\right) ^{2}}$ $,$ $x\in \left[ a,b\right] .$\bigskip
\end{theorem}

\begin{proof}
Since $f$,$g\in MT(I),$ we have%
\begin{eqnarray*}
f\left( ta+\left( 1-t\right) b\right) &\leq &\frac{\sqrt{t}}{2\sqrt{1-t}}%
f\left( a\right) +\frac{\sqrt{1-t}}{2\sqrt{t}}f\left( b\right) \\
g\left( ta+\left( 1-t\right) b\right) &\leq &\frac{\sqrt{t}}{2\sqrt{1-t}}%
g\left( a\right) +\frac{\sqrt{1-t}}{2\sqrt{t}}g\left( b\right)
\end{eqnarray*}%
Since $f$ and $g$ are nonnegative, we write that%
\begin{eqnarray*}
&&f\left( ta+\left( 1-t\right) b\right) g\left( ta+\left( 1-t\right) b\right)
\\
&\leq &\frac{t}{4\left( 1-t\right) }f\left( a\right) g\left( a\right) +\frac{%
1}{4}\left( f\left( a\right) g\left( b\right) +f\left( b\right) g\left(
a\right) \right) +\frac{1-t}{4t}f\left( b\right) g\left( b\right) \\
&=&\frac{t^{2}}{4t\left( 1-t\right) }f\left( a\right) g\left( a\right) +%
\frac{t\left( 1-t\right) }{4t\left( 1-t\right) }\left( f\left( a\right)
g\left( b\right) +f\left( b\right) g\left( a\right) \right) +\frac{\left(
1-t\right) ^{2}}{4t\left( 1-t\right) }f\left( b\right) g\left( b\right)
\end{eqnarray*}%
Consequently,%
\begin{eqnarray*}
&&t\left( 1-t\right) f\left( ta+\left( 1-t\right) b\right) g\left( ta+\left(
1-t\right) b\right) \\
&\leq &\frac{1}{4}\left\{ t^{2}f\left( a\right) g\left( a\right) +t\left(
1-t\right) \left[ f\left( a\right) g\left( b\right) +f\left( b\right)
g\left( a\right) \right] +\left( 1-t\right) ^{2}f\left( b\right) g\left(
b\right) \right\}
\end{eqnarray*}%
By integrating, we get%
\begin{eqnarray*}
&&\frac{1}{b-a}\int_{a}^{b}\frac{\left( b-x\right) \left( x-a\right) }{%
\left( b-a\right) ^{2}}f\left( x\right) g\left( x\right) dx \\
&\leq &\frac{1}{12}\left[ f\left( a\right) g\left( a\right) +f\left(
b\right) g\left( b\right) \right] +\frac{1}{24}\left[ f\left( a\right)
g\left( b\right) +f\left( b\right) g\left( a\right) \right] .
\end{eqnarray*}%
The proof is completed.
\end{proof}

\begin{remark}
\bigskip If we choose $x=\frac{a+b}{2}$ in the inequality (\ref{x}), we
obtain the special state of the inequality (\ref{PA}).
\end{remark}

\begin{theorem}
\bigskip Let $f,g$ be similarly ordered, nonnegative and $MT-convex$
functions on $I$, $a,b\in I$ with $a<b$ and $fg\in L_{1}\left[ a,b\right] .$%
Then we have the inequality%
\begin{equation*}
\frac{1}{b-a}\int_{a}^{b}\mu \left( x\right) f\left( x\right) g\left(
x\right) dx\leq \frac{f\left( a\right) g\left( a\right) +f\left( b\right)
g\left( b\right) }{8}
\end{equation*}%
where $\mu \left( x\right) =\frac{\left( b-x\right) \left( x-a\right) }{%
\left( b-a\right) ^{2}}$ $,$ $x\in \left[ a,b\right] .$\bigskip
\end{theorem}

\begin{proof}
Since $f$,$g\in MT(I),$ we have%
\begin{eqnarray*}
f\left( ta+\left( 1-t\right) b\right) &\leq &\frac{\sqrt{t}}{2\sqrt{1-t}}%
f\left( a\right) +\frac{\sqrt{1-t}}{2\sqrt{t}}f\left( b\right) \\
g\left( ta+\left( 1-t\right) b\right) &\leq &\frac{\sqrt{t}}{2\sqrt{1-t}}%
g\left( a\right) +\frac{\sqrt{1-t}}{2\sqrt{t}}g\left( b\right)
\end{eqnarray*}%
Since $f$ and $g$ are nonnegative and similarly ordered, we write that%
\begin{eqnarray*}
&&f\left( ta+\left( 1-t\right) b\right) g\left( ta+\left( 1-t\right) b\right)
\\
&\leq &\frac{t}{4\left( 1-t\right) }f\left( a\right) g\left( a\right) +\frac{%
1}{4}\left( f\left( a\right) g\left( b\right) +f\left( b\right) g\left(
a\right) \right) +\frac{1-t}{4t}f\left( b\right) g\left( b\right) \\
&=&\frac{t^{2}}{4t\left( 1-t\right) }f\left( a\right) g\left( a\right) +%
\frac{t\left( 1-t\right) }{4t\left( 1-t\right) }\left( f\left( a\right)
g\left( b\right) +f\left( b\right) g\left( a\right) \right) +\frac{\left(
1-t\right) ^{2}}{4t\left( 1-t\right) }f\left( b\right) g\left( b\right) \\
&\leq &\frac{t^{2}}{4t\left( 1-t\right) }f\left( a\right) g\left( a\right) +%
\frac{t\left( 1-t\right) }{4t\left( 1-t\right) }\left( f\left( a\right)
g\left( a\right) +f\left( b\right) g\left( b\right) \right) +\frac{\left(
1-t\right) ^{2}}{4t\left( 1-t\right) }f\left( b\right) g\left( b\right)
\end{eqnarray*}%
Consequently,%
\begin{eqnarray*}
&&4t\left( 1-t\right) f\left( ta+\left( 1-t\right) b\right) g\left(
ta+\left( 1-t\right) b\right) \\
&\leq &t^{2}f\left( a\right) g\left( a\right) +t\left( 1-t\right) \left[
f\left( a\right) g\left( a\right) +f\left( b\right) g\left( b\right) \right]
+\left( 1-t\right) ^{2}f\left( b\right) g\left( b\right) \\
&=&tf\left( a\right) g\left( a\right) +\left( 1-t\right) f\left( b\right)
g\left( b\right)
\end{eqnarray*}%
By integrating, we get%
\begin{equation*}
\frac{4}{b-a}\int_{a}^{b}\mu \left( x\right) f\left( x\right) g\left(
x\right) dx\leq \frac{1}{2}\left[ f\left( a\right) g\left( a\right) +f\left(
b\right) g\left( b\right) \right] .
\end{equation*}%
The proof is completed.
\end{proof}

\bigskip

\end{document}